\documentclass{amsart}
\usepackage[latin1]{inputenc}
\usepackage[english]{babel}
\usepackage{indentfirst}
\usepackage{amssymb}
\usepackage{amsthm}
\usepackage{xcolor}
\usepackage[all]{xy}

\DeclareMathAlphabet{\mathpzc}{OT1}{pzc}{m}{it}
\DeclareMathAlphabet{\mathcal}{OMS}{cmsy}{m}{n}

\newtheorem{THEO}{Theorem}[section]

\newtheorem{CORO}[THEO]{Corollary}
\newtheorem{LEMM}[THEO]{Lemma}

\newtheorem{DEFI}[THEO]{Definition}

\def\({\left(}
\def\){\right)}
\def\diam{\text{$\operatorname{diam}$}}
\def\N{\mathbb{ N}}
\def\R{\mathbb{ R}}

\def\rto{\longrightarrow}

\def\CC{\mathcal{C}}

\def\1{\textbf{1}}

\def\w{\omega}

\author{Gonzalo Mart\'{i}nez-Cervantes}
\address{Departamento de Matem\'{a}ticas\\ Facultad de Matem\'{a}ticas\\ Universidad de Murcia\\ 30100 Espinardo (Murcia)\\ Spain} 
\email{gonzalo.martinez2@um.es}
\title{On weakly Radon-Nikod\'{y}m compact spaces}

\subjclass[2010]{46B22, 46B50, 54G20}

\keywords{Radon-Nikod\'{y}m compact, weakly Radon-Nikod\'{y}m compact, quasi weakly Radon-Nikod\'{y}m}

\thanks{This work was supported by the research project 19275/PI/14 funded by Fundaci\'{o}n S\'{e}neca - Agencia de Ciencia y Tecnolog\'{i}a de la Regi\'{o}n de Murcia within the framework of PCTIRM 2011-2014.  This work was also supported by Ministerio de Econom\'{i}a y Competitividad and FEDER (project MTM2014-54182-P)}

\begin{document}

\begin{abstract}

A compact space is said to be weakly Radon-Nikod\'{y}m if it is homeomorphic to a weak$^\ast$-compact subset of the dual of a Banach space not containing an isomorphic copy of $\ell_1$.
In this paper we provide an example of a continuous image of a Radon-Nikod\'{y}m compact space which is not weakly Radon-Nikod\'{y}m. Moreover, we define a superclass of the continuous images of weakly Radon-Nikod\'{y}m compact spaces and study its relation with Corson compacta and weakly Radon-Nikod\'{y}m compacta.

\end{abstract}

\maketitle

In \cite{MTK:7010772}, I. Namioka defined a compact space $K$ to be \textbf{Radon-Nikod\'{y}m} (RN for short) if and only if it is homeomorphic to a weak$^\ast$-compact subset of a dual Banach space with the Radon-Nikod\'{y}m property. One of the most important questions posed by Namioka in this paper was whether the class of RN compact spaces is closed under continuous images. This question was solved negatively by A. Avil\'{e}s and P. Koszmider in \cite{Aviles2013}.

The class of weakly Radon-Nikod\'{y}m compact spaces generalizes the class of RN compact spaces.
In \cite{GlasnerMegrelishvili2012}, E. Glasner and M. Megrelishvili define a compact space to be \textbf{weakly Radon-Nikod\'{y}m} (WRN for short) if and only if it is homeomorphic to a weak$^\ast$-compact subset of the dual of a Banach space not containing an isomorphic copy of $\ell_1$.
In \cite{2014arXiv1405.2588G}, they ask whether the class of WRN compact spaces is stable under continuous images.
In these papers, a picture of this class of compact spaces is given and it is characterized in terms of Banach spaces of continuous functions. They prove that linearly ordered compact spaces are WRN. Moreover, 
\cite{2014arXiv1405.2588G} contains a proof by S. Todorcevic that $\beta \N$ is not WRN.

In this work we answer negatively the question of E. Glasner and M. Megrelishvili by proving that a modification of the construction given in \cite{Aviles2013} provides an example of a continuous image of a RN compact space which is not WRN.
We construct a RN compact space $\mathbb{L}_0$, a non-WRN compact space $\mathbb{L}_1$ and a surjective continuous function $\pi: \mathbb{L}_0 \rto \mathbb{L}_1$ in the same way as in \cite{Aviles2013}.

In the second section we define quasi WRN compact spaces, a superclass of the continuous images of WRN compact spaces. We prove that zero-dimensional quasi WRN compact spaces are WRN. We also show other relations of quasi WRN compacta with Corson and Eberlein compacta, including an example of a Corson compact space which is not quasi WRN.

\section{A continuous image of a RN compact space which is not WRN}

\subsection{Construction of $\mathbb{L}_0$, $\mathbb{L}_1$ and $\pi : \mathbb{L}_0 \rto \mathbb{L}_1$ }

Any unexplained notation can be found in \cite{Aviles2013}. By $\Delta = 2^\N= \lbrace 0,1 \rbrace^\N $ we denote the Cantor set with the topology induced by the metric $\rho : \Delta \times \Delta \rto \R$ given by $\rho (x,y)= 2 ^{-\min \lbrace k: x_k \neq y_k \rbrace }$ if $x \neq y$. 
$T=2^{<\w}$ denotes the set of all finite sequences of 0's and 1's and, for every $t\in T$, $|t|$ denotes the length of $t$. By $G$ we denote all finite sets $g=\lbrace s_1, ..., s_n \rbrace $ such that 
$s_1,...,s_n \in T$ verify $|s_1|=...=|s_n|$.
Given $t\in T$ and $g \in G$ we define $\Gamma_{g}^t: \Delta \rto \Delta$ as:
\begin{itemize}
\item $\Gamma_g^t ( s ^\frown \lambda) = t ^\frown \lambda $ if $s \in g$ and $\lambda \in \Delta$,
\item $\Gamma_g^t ( z )=t^\frown (0,0,...)$ in the rest of points.
\end{itemize}
Notice that each $\Gamma_g^t$ is continuous.

The main difference between our construction and that of \cite{Aviles2013} is the choice of the functions $\Gamma_g^t$ with $g \in G$ and $t \in T$.
This choice will let us prove that $\mathbb{L}_1$ is not WRN.

Let $K=\bigcup_{t\in T} A_t \cup B \cup C$ be a compact space such that:
\begin{enumerate}
\item[(1)] All points of $A=\bigcup_{t\in T} A_t$ are isolated in $K$.
\item[(2)] For every $x \in B$ there exists an infinite set $C_x \subset A$ such that $\overline{C_x}=C_x \cup \lbrace x \rbrace$ and moreover, $\overline{ C_x}$ is open in $K$.
\item[(3)] There exists a function $\psi : B \rto G^T$ such that for any family  of subsets of $A$ of the form $\lbrace X_g^t: g \in G, t \in T \rbrace$ with $A_t = \bigcup_{g\in G} X_g^t$ for every $t \in T$, there exists $x \in B$ such that  $C_x \cap X^t_{\psi(x)[t]}$ is infinite for all $t\in T$.

\end{enumerate}

A. Avil\'{e}s and P. Koszmider called a compact space of the previous form a basic space and they provided some examples of such compact spaces.

Consider $L=(A \times \Delta ) \cup B \cup C$, $\mathbb{L}_0$, $\mathbb{L}_1$ and $\pi: \mathbb{L}_0 \rto \mathbb{L}_1$ defined in the same way as in \cite{Aviles2013}:

\begin{itemize}
\item A basic neighborhood of a point $(a,t)$ in $L$ is of the form $\lbrace a \rbrace \times U $ where $U$ is a neighborhood of $t$ in $\Delta$. A basic neighborhood of a point $x \in B \cup C$ is of the form $((U \cap A) \times \Delta) \cup (U \setminus A )$, where $U$ is a neighborhood of $x$ in $K$.

\item $q: \Delta \rto [0,1]$ is the continuous surjective function given by the formula $q(t_1, t_2,...)= \sum_{k \in \N} \frac{t_k}{2^k} $.

\item For every $x \in B$, $g_x: L \setminus \lbrace x \rbrace \rto \Delta $ is the continuous function given by the formula $g_x(a,z)=\Gamma_{\psi(x)[t]}^t (z)$ for $a\in A_t \cap C_x$, $z\in \Delta$ and $g_x(y)=0$ in any other case.

\item For every $x \in B$, $f_x: L\setminus \lbrace x \rbrace \rto [0,1] $ is the continuous function $f_x = q \circ g_x $.

\item $\mathbb{L}_0 = \lbrace [u,v] \in L \times \Delta^B: g_x(u)=v_x \mbox{ for all } x\in B\setminus \lbrace u \rbrace \rbrace$.

\item $\mathbb{L}_1 = \lbrace [u,v] \in L \times [0,1]^B: f_x(u)=v_x \mbox{ for all } x\in B\setminus \lbrace u \rbrace \rbrace$.

\item $\pi [u,v] = [u , q\(v_x\)_{x \in B} ]$.
\end{itemize}

It is clear that $\pi$ is continuous and surjective.
$\mathbb{L}_0$ is RN, since the proof given in \cite{Aviles2013} also works in this case because it does not use the definition of the functions $\Gamma_s^t$, just its continuity and the fact that the image of $\Gamma_s^t$ is $\lbrace t^\frown \lambda: \lambda \in \Delta \rbrace$.

\subsection{$\mathbb{L}_1$ is not WRN}

For any compact space $X$, $\CC(X)$ is the Banach space 	of real valued continuous functions on $X$.

\begin{DEFI}
Let $X$ be a compact space. A family of functions $F \subset \CC(X)$ is said to be \emph{fragmented} if for every nonempty subset $A$ of $X$ and every $\varepsilon >0$ there exists an open subset $O$ in $X$ such that $O\cap A$ is nonempty and $f(O\cap A)$ has diameter smaller than $\varepsilon$ for every $f \in F$.
$F$ is said to be \emph{eventually fragmented} if every sequence in $F$ has a subsequence which is a fragmented family on $X$.
\end{DEFI}

We will use the following characterization of WRN compact spaces:

\begin{THEO}[{\cite[Theorem 6.5]{GlasnerMegrelishvili2012}}]
\label{GlasnerMegrelishvili}

Let $X$ be a compact space. Then, $X$ is WRN if and only if there exists an eventually fragmented uniformly bounded family of continuous functions $F \subset \CC(X)$  which separates the points of $X$. 
\end{THEO}

The proof of the previous theorem is based on the construction of Davis-Figiel-Johnson-Pelczy\'{n}ski \cite{DAVIS1974311} and Rosenthal's $\ell_1$-theorem \cite{zbMATH03465962}.
\begin{THEO}
$\mathbb{L}_1$ is not WRN.
\end{THEO}

\begin{proof}

Suppose that $F \subset \CC( \mathbb{L}_1)$ is an eventually fragmented uniformly bounded family. We will find two points that are not separated by $F$.
For every $a \in A$ and every $z_1 \leq  z_2 \in \Delta$, we denote by $a+z_1$ the unique point of $\mathbb{L}_1$ of the form $[(a,z_1), v]$ and by $a+[z_1,z_2]$ the set of points $a+z$ with $z\in [z_1,z_2]$, where the order in $\Delta$ is the lexicographical order. Similarly, for every $x \in B$ and every $\xi_1 \leq \xi_2 \in [0,1]$ we denote by $x \oplus \xi_1 = [x, v] \in \mathbb{L}_1$ the point given by the formula $v_y= f_y (x)$ for every $y \in B \setminus \lbrace x \rbrace $ and $v_x= \xi_1$ and by $x \oplus [ \xi_1, \xi_2] $ we denote the set of points $x \oplus \xi$ with $\xi \in [\xi_1, \xi_2]$.
For every $a \in A_t$ and $f \in F$, we can find $s_f (a) \in T$ such that 
$$ \diam \( f(a+[s_f(a)^\frown (0,0,...) , s_f(a) ^\frown (1,1,...) ] ) \) < \frac{1}{4^{|t|}} .$$
For each $a \in A$ and $f \in F$ fix $s_f (a)$ with the previous property and $|s_f(a)|$ minimum.

Then, the set $\lbrace s_f(a): f \in F \rbrace $ is finite for every $a \in A$.
Namely, if this set were not finite, there would exist $a\in A_t$ for some $t\in T$ and a sequence $\lbrace f_n \rbrace_{n \in \N}$ in $F$ such that $|s_{f_n}(a)| \rto \infty $, so $\lbrace f_n : n \in \N \rbrace $ would not have a fragmented subsequence, since for any open subset $O$ of $a+\Delta$ there would exist $N$ such that $f_n(O)$ has diameter bigger than $\frac{1}{4^{|t|}} $ for every $n > N$.
Therefore, $\lbrace s_f(a): f \in F \rbrace $ is finite for every $a \in A$ due to the eventually fragmentability of $F$. Thus, for every $a \in A_t$ there exists $g_a \in G$ such that for every $f \in F$ we can find $s \in g_a$ with $$ \diam \( f(a+[s^\frown (0,0,...) , s^\frown (1,1,...) ] ) \) < \frac{1}{4^{|t|}} .$$ 
Let $X^t_g= \lbrace a \in A_t: g_a = g \rbrace$ for every $t \in T$ and every $g \in G$.
These sets verify $A_t = \bigcup _{g \in G} X^t_g$ for every $t \in T$. Due to property (3), there exists $x \in B$ such that $C_x \cap X^t_{\psi(x)[t] }$ is infinite for every $t \in T$. 

We are going to prove that $F$ does not separate the points of $\mathbb{L}_1$ by showing that $f(x \oplus 0 )= f( x \oplus 1)$ for every $f \in F$.
Fix $f \in F$ and an infinite subset $\lbrace a_n : n \in \N \rbrace \subset C_x \cap X^t_{\psi(x)[t] } $.
Since $g_{a_n}= \psi (x)[t] \in G$ for every $n \in \N$ and $\psi(x) [t]$ is finite, there exist a subsequence $\lbrace a_{n_k} \rbrace_{k \in \N}$ and $s\in \psi(x)[t]$ such that $$ \diam \( f(a_{n_k}+[s^\frown (0,0,...) , s^\frown (1,1,...) ] ) \) < \frac{1}{4^{|t|}} \mbox{ for every } k \in \N .$$

Notice that
$$ f_x(a_{n_k}+s^\frown (i,i,...))= q(\Gamma_{\psi(x)[t]}^t (s ^\frown (i,i,...))=q(t ^\frown (i,i,...))=:t^i $$
for every $i \in \lbrace 0,1 \rbrace $. Taking limits we obtain $a_{n_k}+s^\frown (i,i,...) \rto x \oplus \xi^i$ for every $i \in \lbrace 0,1 \rbrace $, where 
$$ \xi^i = \lim_n f_x( a_{n_k} + s^\frown (i,i,...))= t^i .$$  

For every $ \xi_0, \xi_1 \in [t^0, t^1]$, there exist $\lambda_1, \lambda_2 \in \Delta$ such that $q( t ^\frown \lambda_i )=\xi_i$ and therefore $a_{n_k}+ s ^\frown \lambda_i  \rto x \oplus \xi_i $, so $$d( f(x \oplus \xi_0), f(x \oplus \xi_1) )= \lim_n d( f(a_{n_k}+ s ^\frown \lambda_0) , f(a_{n_k}+ s ^\frown \lambda_1)) \leq \frac{1}{4^{|t|}} .$$
Thus, 
$$ \diam \( f( x \oplus [ t^0, t^1] ) \) \leq \frac{1}{4^{|t|}} .$$

Now, since for every $m \in \N$ 
$$ \lbrace [t^0, t^1]: t \in T, |t|=m  \rbrace = \lbrace [(k-1)2^{-m}, k 2^{-m}]: k=1,2,..., 2^m \rbrace ,$$ 
it follows that $\diam f( x \oplus [ 0, 1] ) \leq 2^m \frac{1}{4^{m}}= \frac{1}{2^m}$ for every $m \in \N$.
Therefore, $f(x \oplus 0 ) = f( x \oplus 1) $ and $F$ does not separate $x \oplus 0$ and $x \oplus 1$.
\end{proof}

\section{WRN and QWRN compact spaces}

\begin{DEFI}
A sequence $\(A_n ^0, A_n ^1\)_{n \in \N}$ of disjoint pairs of subsets of a set $S$ is said to be \textbf{independent} if for every $n\in \N$ and every $\varepsilon : \lbrace 1,2,...,n \rbrace \rto \lbrace 0,1 \rbrace $ we have
$\bigcap_{k=1}^n A_k^{\varepsilon(k)} \neq \emptyset$.

A sequence of functions $\(f_n\)_{n \in \N} \subset \R ^S$ is said to be \textbf{independent} if there exist real numbers $p<q$ such that the sequence $\(A_n ^0, A_n ^1\)_{n \in \N}$ is independent, where $A_n ^0= \lbrace s \in S: f_n (s)<p \rbrace $ and $A_n ^1= \lbrace s \in S: f_n (s)>q \rbrace $ for every $n\in \N$.

\end{DEFI}

Every compact space $K$ is homeomorphic to a subspace of a product space $[0,1]^\Gamma$. If $i: K \rto [0,1]^\Gamma$ is an embedding, then $x_\alpha$ denotes $\alpha (i(x))$ for every $x \in K$ and every $\alpha \in \Gamma$.

Using Rosenthal's $\ell_1$-theorem and the fact that a uniformly bounded family of continuous functions is eventually fragmented if and only if it does not contain a sequence equivalent to the unit basis of $\ell_1$ (see \cite[Fact 4.3 and Proposition 4.6]{GlasnerMegrelishvili2012}), Theorem \ref{GlasnerMegrelishvili} can be reformulated in the following way:

\begin{THEO}
\label{WRNindependent}
A compact space $K$ is WRN if and only if there exists an homeomorphic embedding of $K$ in $ [0,1]^\Gamma$ such that for every $p<q$, the family of disjoint pairs of subsets $\(A_\alpha^0, A_\alpha^ 1\)_{\alpha \in \Gamma}$ does not contain independent sequences, where $A_\alpha^0 = \lbrace x \in K: x_\alpha < p \rbrace$ and $A_\alpha^1 = \lbrace x \in K: x_\alpha > q \rbrace$ for every $\alpha \in \Gamma$.
\end{THEO}
\begin{proof}

If $K$ is WRN, then there exists an eventually fragmented uniformly bounded family $F \subset \CC(K)$ separating the points of $K$. Let $\| F \| := \sup_{f \in F} \|f \|$ and consider $\Gamma := \frac{F+\|F\|}{2\|F\|}$. We know that $\Gamma$ is an eventually fragmented family of continuous functions in $[0,1]^K$ separating the points of $K$. The function $i: K \rto [0,1]^\Gamma$ given by the formula $i(x)=\( \alpha(x) \)_{\alpha \in \Gamma}$ is an homeomorphic embedding. Since $\Gamma$ does not contain a sequence equivalent to the unit basis of $\ell_1$, it follows from \cite[Proposition 4]{zbMATH03465962} that for every real numbers $p<q$ the family of pairs $\(A_\alpha^0, A_\alpha^ 1\)_{\alpha \in \Gamma}$ does not contain independent sequences, where $A_\alpha^0 = \lbrace x \in K: x_\alpha = \alpha(i(x)) < p \rbrace$ and $A_\alpha^1 = \lbrace x \in K: x_\alpha =\alpha(i(x)) > q \rbrace$ for every $\alpha \in \Gamma$.

On the other hand, suppose $K$  is a subset of $ [0,1]^\Gamma$ and for every real numbers $p<q$ the family of pairs $\(A_\alpha^0, A_\alpha^ 1\)_{\alpha \in \Gamma}$ does not contain independent sequences, where $A_\alpha^0 = \lbrace x \in K: x_\alpha < p \rbrace$ and $A_\alpha^1 = \lbrace x \in K: x_\alpha > q \rbrace$ for every $\alpha \in \Gamma$. Then $\Gamma \subset \CC(K)$ is a family without sequences equivalent to the unit basis of $\ell_1$ and therefore it is an eventually fragmented uniformly bounded family separating the points of $K$. Thus, $K$ is WRN due to Theorem \ref{GlasnerMegrelishvili}.
\end{proof}

\begin{DEFI}
\label{QWRN}
A compact space $K$ is \textbf{quasi WRN} (QWRN for short) if there exists an homeomorphic embedding of $K$ in $[0,1]^\Gamma$ such that for every $\varepsilon >0$ there exists a decomposition $\Gamma=\bigcup_{n \in \N} \Gamma_n^\varepsilon $ such that for every $p<q$ with $q-p> \varepsilon $, the family of pairs $\(A_\alpha^0, A_\alpha ^1\)_{\alpha \in \Gamma_n^\varepsilon}$ does not contain independent sequences for every $n \in \N$, where $A_\alpha^0 = \lbrace x \in K: x_\alpha < p \rbrace$ and $A_\alpha^1 = \lbrace x \in K: x_\alpha > q \rbrace$ for every $\alpha \in \Gamma$.
\end{DEFI}

In \cite{zbMATH01709966} and \cite{zbMATH01450217} two superclasses of continuous images of RN compacta are defined.
In \cite{zbMATH02120157} it is proved that both superclasses are equal. Compact spaces of these superclasses are called QRN.
In \cite{zbMATH01709966} it is proved that zero-dimensional QRN compact spaces are RN and that continuous images of QRN compact spaces are QRN.  
In essence, our definition of QWRN is analogous to the definition given in \cite{zbMATH01450217}. In this section, we prove similar results for QWRN compact spaces.

From Theorem \ref{WRNindependent} it follows that every WRN compact space is QWRN.
A useful characterization of QWRN compact spaces is given by the following lemma:

\begin{LEMM}
\label{QWRN2}
A compact space $K$ is QWRN if and only if there exists an homeomorphic embedding of $K$ in $[0,1]^\Gamma$ satisfying that for every $p<q$ there exists a countable decomposition $\Gamma=\bigcup_{n \in \N} \Gamma_n^{p,q} $ such that the family of pairs $\( A_\alpha^0, A_\alpha^1 \)_{\alpha \in \Gamma_n^{p,q}}$ does not contain independent sequences for every $n\in \N$, where $A_\alpha ^0 = \lbrace x \in K: x_\alpha < p \rbrace$ and  $A_\alpha^1 = \lbrace x \in K: x_\alpha > q \rbrace$ for every $\alpha \in \Gamma$.
\end{LEMM}

\begin{proof}
Set $\varepsilon=q-p$. If $K$ is QWRN, then we can take $\Gamma_n^{p,q} = \Gamma_n^{\frac{\varepsilon}{2}}$ for every $n \in \N$.

Now, fix $\varepsilon>0$. There exist $p_1 < p_2 < ... < p_m$ such that for every $p<q $ with $q-p>\varepsilon $, there exist $p<p_j < p_{j+1} < q$ for some $j \leq m$. Thus, we can obtain a countable decomposition of $$ \Gamma = \bigcup_{(n_1,...,n_{m-1}) \in \N^{m-1}} \bigcap_{j=1}^{m-1} \Gamma_{n_j} ^{p_j,p_{j+1}},$$
with each $ \bigcap_{j=1}^{m-1} \Gamma_{n_j} ^{p_j,p_{j+1}}$ verifying that for every $p<q$ with $q-p> \varepsilon $, the family of pairs $\(A_\alpha^0, A_\alpha ^1\)_{\alpha \in \bigcap_{j=1}^{m-1} \Gamma_{n_j} ^{p_j,p_{j+1}}}$ does not contain independent sequences, where $A_\alpha^0 = \lbrace x \in K: x_\alpha < p \rbrace$ and $A_\alpha^1 = \lbrace x \in K: x_\alpha > q \rbrace$ for every $\alpha \in \Gamma$.
\end{proof}

The following lemma is a modification of Lemma 3 in \cite{zbMATH03465962}.

\begin{LEMM}[{\cite[Lemma~9.5]{2014arXiv1405.2588G}}]
\label{IndSeq}
Let $\(A_n^0, A_n^1\)_{n \in \N}$ be an independent sequence of disjoint pairs of subsets of a set $S$. Suppose there exist $N \in \N$ and $N$ sequences of disjoint pairs $(A_{n,j}^0, A_{n,j}^1)_{n \in \N}$ with $j=1,2,...,N$ such that 
$$ A_n^0 \times A_n^1 \subset \bigcup_{j=1}^N A_{n,j}^0 \times A_{n,j}^1  \mbox{ for every } n \in \N .$$ 

Then, there is $j_0 \in \lbrace 1,2,...,N \rbrace$ and a subsequence $\(n_k \)_{k\in \N}$ of $\N$ such that $\(A_{n_k ,j_0}^0, A_{n_k ,j_0}^1\)_{k \in \N}$ is an independent sequence.

\end{LEMM}

\begin{THEO}
\label{continuousimageqwrn}

The continuous image of a QWRN compact space is QWRN.

\end{THEO}

\begin{proof}

Let $f:L \rto K$ be a continuous surjective function with $K \subset [0,1]^\Gamma$, $L \subset [0,1]^ \Lambda$ and $\Lambda$ verifying the conditions of Definition \ref{QWRN}. We are going to prove that $\Gamma$ verifies the conditions of Lemma \ref{QWRN2}.

Fix $p<p'<q'<q$ and $A_\alpha^0 = \lbrace x \in K: x_\alpha < p \rbrace$, $A_\alpha^1 = \lbrace x \in K: x_\alpha > q \rbrace$ for every $\alpha \in \Gamma$.
A basis for the topology of $L$ is given by the open sets $$U_{(\beta, r, s)}= \lbrace y \in L: r_i < y_{\beta_i} < s_i \mbox{ for each } i=1,...,n \rbrace $$ with $\beta=(\beta_1, ...,\beta_n) \in \Lambda^n$, $r=(r_1,...,r_n),s=(s_1,...,s_n) \in [-1,2]^n$ and $n\in \N$.
Therefore, $$f^{-1}(A_\alpha^0)\subset f^{-1}(\lbrace x \in K: x_\alpha \leq p \rbrace) \subset f^{-1}( \lbrace x \in K: x_\alpha < p' \rbrace )=: \bigcup_{(\beta, r, s)\in S_\alpha '} U_{(\beta, r,s)}$$ for some set $S_\alpha'$. Due to the compactness of $f^{-1}(\lbrace x \in K: x_\alpha \leq p \rbrace)$, there exists a finite set $S_\alpha \subset S_\alpha '$ such that 
$f^{-1}(A_\alpha^0)\subset  \bigcup_{(\beta, r, s)\in S_\alpha} U_{(\beta, r,s)}$. 
Similarly, there exists a finite set $S_\alpha^1$ such that 
$$ f^{-1}(A_\alpha^1)\subset  \bigcup_{(\beta, r, s)\in S_\alpha^1} U_{(\beta, r,s)} \subset f^{-1}(\lbrace x \in K: x_\alpha > q' \rbrace .$$

Without loss of generality, we can take for every $\alpha \in \Gamma $ a natural number $n_\alpha $ such that $|S_\alpha | = |S_\alpha ^1 | = n_\alpha$. Set
$$ \lbrace U_{(\beta, r,s)} : (\beta, r, s)\in S_\alpha \rbrace =: \lbrace U_i^{\alpha,0} : i=1,2,...,n_\alpha \rbrace $$ and $$ \lbrace U_{(\beta, r,s)} : (\beta, r, s)\in S_\alpha^1 \rbrace =: \lbrace U_i^{\alpha,1} : i=1,2,...,n_\alpha \rbrace .$$

For every $U_{(\beta, r, s)}$ and $m \in \N$, we define
$$ m(U_{(\beta, r,s)}) := \left\{ y \in L:  y_{\beta_i} < r_i-\frac{1}{m} \mbox{ or } y_{\beta_i} > s_i+\frac{1}{m} \mbox{ for some } i \right\} .$$
Notice that $\overline{ U_i^{\alpha, 0} } \cap \overline{ U_j^{\alpha, 1} }= \emptyset $ for every $i,j =1,2,...,n_\alpha$. Therefore, for each $\alpha \in \Gamma$, we can fix $m_\alpha \in \N$ such that 
$$ U_j^{\alpha, 1} \subset m_\alpha (U_i^{\alpha, 0}) \mbox{ for every } i,j=1,2,...,n_\alpha .$$
For every $\varepsilon > 0$, we have a decomposition $\Lambda = \bigcup_n \Lambda_n ^ \varepsilon$ with each $\Lambda_n^ \varepsilon$ verifying the conditions of Definition \ref{QWRN}.
For every $n,m,N \in \N$, define $ \Gamma_{n,m,N} \subset \Gamma$ the set of all points $\alpha \in \Gamma$ with $n_\alpha=n$, $m_\alpha = m$, $|\beta| \leq N$ and $\beta_i \in \bigcup_{k=1}^N \Lambda_k^{\frac{1}{2m}}$ for every $(\beta, r, s) \in S_\alpha$.

We claim that $\(A_\alpha^0, A_\alpha^1\)_{\alpha \in \Gamma_{n, m, N}}$ has an independent sequence. If not, the family $\(f^{-1}\(A_\alpha^0\), f^{-1}\(A_\alpha^1\)\)_{\alpha \in \Gamma_{n, m, N}}$ contains an independent sequence too.
Since $$f^{-1}(A_\alpha^0) \times f^{-1}(A_\alpha^1) \subset \bigcup_{i,j=1,...,n} U_i^{\alpha,0} \times U_j^{\alpha,1} ,$$ there exist $i,j \in \lbrace 1,2,...,n \rbrace$ such that the family $\(U_i^{\alpha,0}, U_j^{\alpha,1}\)_{\alpha \in \Gamma_{n,m, N}}$ contains an independent sequence, due to Lemma \ref{IndSeq}.
Therefore, $\(U_i^{\alpha,0}, m\(U_i^{\alpha,0}\)\)_{\alpha \in \Gamma_{n,m, N}}$ contains an independent sequence.
By definition,
$$ m(U_i^{\alpha,0}) = \bigcup_{t=1,...,k} \left\{ y \in L: y_{\beta_t^\alpha} < r_t^\alpha-\frac{1}{m} \right\} \cup \left\{ y \in L: y_{\beta_t^\alpha} > s_t^\alpha + \frac{1}{m} \right\} ,$$
where $U_i^{\alpha, 0} = U_{(\beta,r,s)}$ and $\beta=(\beta_1^\alpha,...,\beta_k^\alpha)$, $r=(r_1^\alpha,...,r_k^\alpha)$ and $s=(s_1^\alpha, ..., s_k^\alpha)$.
Without loss of generality, we suppose that 
$$\(U_i^{\alpha,0}, \left\lbrace y \in L: y_{\beta_t^\alpha} <r_t^\alpha-\frac{1}{m} \right\rbrace  \)_{\alpha \in \Gamma_{n,m, N}}$$ contains an independent sequence. Since $U_i^{\alpha, 0} \subset \left\{ y \in L: y_{\beta_t^\alpha} >r_t^\alpha \right\} $,
the family of pairs $\(\left\{ y \in L: y_{\beta_t^\alpha} >r_t^\alpha \right\} , \left\{ y \in L: y_{\beta_t^\alpha} <r_t^\alpha-\frac{1}{m} \right\} \)_{\alpha \in \Gamma_{n,m, N}}$ contains an independent sequence.
Therefore, there exists an independent sequence of the form $\(\left\{ y \in L: y_{\beta_k} >r_k \right\}, \left\{ y \in L: y_{\beta_k} <r_k-\frac{1}{m} \right\} \)_{k\in \N}$ with $\beta_k \in \Lambda_1^ {\frac{1}{2m}} \cup \Lambda_2^ {\frac{1}{2m}} \cup... \cup \Lambda_N^ {\frac{1}{2m}}$ for every $k \in \N$.
Taking a subsequence if necessary, we can suppose that there exist $t \in \lbrace 1,2,...,N \rbrace$ such that $\beta_k \in \Lambda_t^ {\frac{1}{2m}}$ and that $r_k$ tends to $r \in [-1,2]$.
Therefore, $$\(\left\{ y \in L: y_{\beta} >r-\frac{1}{8m} \right\} , \left\{ y \in L: y_{\beta} <r-\frac{3}{4m} \right\} \)_{\beta \in \Lambda_t^ {\frac{1}{2m}}}$$ contains an independent sequence.
This is a contradiction with the definition of $\Lambda_t^{\frac{1}{2m}}$ because $ r-\frac{1}{8m}-(r-\frac{3}{4m} )= \frac{5}{8m} > \frac{1}{2m} $.

Thus, $\Gamma= \bigcup_{n,m,N \in \N } \Gamma_{n,m,N}$ and $\(A_\alpha^0,A_\alpha^1\)_{\alpha \in \Gamma_{n,m,N}}$ does not contain independent sequences for any $n,m,N \in \N$.  
\end{proof}

Notice that in the previous proof we have not made any assumption on the family $\Gamma$. Therefore, in particular, we have proved the following result: 
\begin{LEMM}
\label{independence}

A compact space $K$ is QWRN if and only if \textbf{for every} homeomorphic embedding of $K$ in $[0,1]^\Gamma$  and for every $\varepsilon >0$ there exists a countable decomposition $\Gamma=\bigcup_{n \in \N} \Gamma_n^\varepsilon $ such that for every $p<q$ with $q-p> \varepsilon $, the family of pairs $\(A_\alpha^0, A_\alpha ^1\)_{\alpha \in \Gamma_n^\varepsilon}$ does not contain independent sequences for every $n \in \N$, where $A_\alpha^0 = \lbrace x \in K: x_\alpha < p \rbrace$ and $A_\alpha^1 = \lbrace x \in K: x_\alpha > q \rbrace$ for every $\alpha \in \Gamma$.

\end{LEMM}

As in the case of quasi RN and RN compact spaces, every zero-dimensional QWRN compact space is WRN.

\begin{THEO}

\label{zerodimensionalQWRN}

Let $K$ be a zero-dimensional QWRN compact space. Then $K$ is WRN.

\end{THEO}

\begin{proof}

Since $K$ is zero-dimensional, we can suppose that  $K \subset \lbrace 0, 1 \rbrace ^\Gamma$ for some set $\Gamma$.
Due to Lemma \ref{independence}, there exists a decomposition $\Gamma = \bigcup_{n \in \N} \Gamma_n $ such that for every $p<q$ with $q-p> \frac{1}{2} $, the family of pairs $A_\alpha^0 = \lbrace x\in K: x_\alpha < p \rbrace$, $A_\alpha^1 = \lbrace x \in K: x_\alpha > q \rbrace$ with $\alpha \in \Gamma_n$ does not contain independent sequences. Since $K \subset \lbrace 0,1 \rbrace ^\Gamma$, we know that the previous family of pairs does not contain independent sequences for any $p<q$.

Let $F=\lbrace f_\alpha \rbrace_{\alpha \in \Gamma} \subset \CC(K)$, where $f_\alpha(x)=\frac{x_\alpha}{n}$ for every $n\in \N$, $\alpha \in \Gamma_n$ and $x\in K$.
The family $F$ separates the points of $K$ and it does not contain an independent sequence of functions, so $K$ is WRN due to Theorem \ref{WRNindependent}.
\end{proof}

As a corollary of Theorems \ref{continuousimageqwrn} and \ref{zerodimensionalQWRN}, we obtain the following result:

\begin{CORO}

Zero-dimensional continuous images of WRN compact spaces are also WRN.

\end{CORO}

We finish the section relating Eberlein and Corson compacta with QWRN and WRN compacta.
Let $\Sigma( \Gamma )$ denote the sigma-product consisting of all $x \in \R^\Gamma$ such that $x$ has countable support. A compact space is said to be Corson if it is homeomorphic to a compact subspace of $\Sigma(\Gamma)$ for some set $\Gamma$ and it is said to be Eberlein if it is homeomorphic to a weak-compact subset of a Banach space.

Farmaki gave the following characterization of those Corson compacta that are Eberlein:

\begin{THEO}[\cite{Farmaki1987}]
A compact space $K \subset \Sigma(\Gamma)$ is Eberlein if and only if for every $\varepsilon >0$ there exist a countable decomposition $\Gamma= \bigcup_{n \in \N} \Gamma_n^\varepsilon $ such that for every $x \in K$ and every $n \in \N$, the set $\lbrace \alpha \in \Gamma_n^\varepsilon: |x_\alpha|>\varepsilon \rbrace$ is finite.

\end{THEO}

The Talagrand's compact is an example of a Corson compact space which is not Eberlein \cite{Talagrand1979}. It consists of all characteristic functions $1_A \in \lbrace 0,1 \rbrace^{\N ^\N}$ with $A \subset \N^\N$ such that there exists $n\in \N$ with 
$x(k)=y(k)$ for every $k=1,2,...,n$  and  $x(n+1) \neq y(n+1) \mbox{ for every } x,y \in A \mbox{ with }x \neq y .$

However, every Corson and RN compact space is Eberlein:

\begin{THEO}[\cite{zbMATH00515838, orihuela1991every}]

A compact $K$ is Eberlein if and only if it is Corson and RN.
\end{THEO}

It is proved in \cite{zbMATH01709966} that the previous theorem can be extended to QRN compact spaces. 
If the compact space is solid, then the previous theorem can also be extended to QWRN compact spaces:

\begin{THEO}

Let $K \subset \Sigma(\Gamma)$ be a solid Corson compact space, where solid means that for every $x\in K$ and every $A \subset \Gamma$ finite, $x1_A \in K$.
Then, $K$ is WRN if and only if it is QWRN if and only if it is Eberlein.
\end{THEO}

\begin{proof}

Since every Eberlein compact space is WRN and every WRN compact space is QWRN, we have to show that if $K$ is QWRN then it is Eberlein.
Suppose $K$ is QWRN and fix $\varepsilon >0$. There exists a decomposition $\Gamma=\bigcup_{n \in \N} \Gamma_n $ such that the family $\(A_\gamma ^0, A_\gamma ^1\)_{\gamma \in \Gamma_n}$  does not contain independent sequences for every $n \in \N$, where $A_\gamma ^0= \lbrace x \in K: x_\gamma > \varepsilon \rbrace $ and 
$A_\gamma ^1 = \lbrace x \in K: x_\gamma < \frac{\varepsilon}{2} \rbrace$.
Let $x \in K$. We are going to see that $\lbrace \gamma \in \Gamma_N: |x_\gamma|>\varepsilon \rbrace$ is finite for every $N \in \N$. Suppose $\( \gamma_n \)_{n \in \N} \subset \Gamma_N$ is a sequence of coordinates with $|x_{\gamma_n}|> \varepsilon$. Since $K$ is solid, for every $\delta : \lbrace 1,2,...,n \rbrace \rto \lbrace 0,1 \rbrace $, the element $x1_{\lbrace \gamma_k: \delta(k)=0 \rbrace }$ is in $K$ and, therefore $(A_{\gamma_n}^0, A_{\gamma_n}^1)$ is independent, since  $$x1_{\lbrace \gamma_k : \delta(k)=0 \rbrace } \in \bigcap_{k=1}^n A_{\gamma_k}^{\delta(k)}.$$
Thus, $\lbrace \alpha \in \Gamma_N : |x_\alpha|>\varepsilon \rbrace$ is finite for every $N \in \N$ and $K$ is Eberlein due to Farmaki's Theorem.
\end{proof}

\begin{CORO}

The Talagrand's compact is solid, Corson and not Eberlein. Therefore, the Talagrand's compact is not QWRN.

\end{CORO}

\section*{Acknowledgements}
I would like to thank Antonio Avil\'{e}s and Jos\'{e} Rodr\'{i}guez for their guidance and the invested time in this work.

\bibliography{WRN}

\begin{thebibliography}{10}

\bibitem{zbMATH01709966}
Alexander~D. {Arvanitakis}.
\newblock {Some remarks on Radon--Nikod\'ym compact spaces.}
\newblock {\em {Fundam. Math.}}, 172(1):41--60, 2002.

\bibitem{zbMATH02120157}
A.~{Avil\'es}.
\newblock {Radon-Nikod\'ym compact spaces of low weight and Banach spaces.}
\newblock {\em {Stud. Math.}}, 166(1):71--82, 2005.

\bibitem{Aviles2013}
A.~Avil\'{e}s and P.~Koszmider.
\newblock A continuous image of a {R}adon-{N}ikod\'{y}m compact space which is
  not {R}adon-{N}ikod\'{y}m.
\newblock {\em Duke Math. J.}, 162(12):2285--2299, 2013.

\bibitem{DAVIS1974311}
W.J. Davis, T.~Figiel, W.B. Johnson, and A.~Pelczynski.
\newblock Factoring weakly compact operators.
\newblock {\em Journal of Functional Analysis}, 17(3):311--327, 1974.

\bibitem{zbMATH01450217}
M.~{Fabian}, M.~{Heisler}, and E.~{Matou\v{s}kov\'a}.
\newblock {Remarks on continuous images of Radon-Nikod\'ym compacta.}
\newblock {\em {Commentat. Math. Univ. Carol.}}, 39(1):59--69, 1998.

\bibitem{Farmaki1987}
V.~Farmaki.
\newblock {The structure of Eberlein, uniformly Eberlein and Talagrand compact
  spaces in $\Sigma ( \mathbb{R} ^\Gamma)$}.
\newblock {\em Fundamenta Mathematicae}, 128(1):15--28, 1987.

\bibitem{GlasnerMegrelishvili2012}
E.~{Glasner} and M.~{Megrelishvili}.
\newblock {Representations of dynamical systems on Banach spaces not containing
  $l_{1}$.}
\newblock {\em {Trans. Am. Math. Soc.}}, 364(12):6395--6424, 2012.

\bibitem{2014arXiv1405.2588G}
E.~{Glasner} and M.~{Megrelishvili}.
\newblock {Eventual nonsensitivity and tame dynamical systems}.
\newblock {\em ArXiv e-prints}, arXiv:1405.2588, May 2014.

\bibitem{MTK:7010772}
I.~Namioka.
\newblock {Radon-Nikod\'{y}m compact spaces and fragmentability}.
\newblock {\em Mathematika}, 34:258--281, 1987.

\bibitem{orihuela1991every}
J.~{Orihuela}, W.~{Schachermayer}, and M.~{Valdivia}.
\newblock {Every Radon-Nikod\'{y}m Corson compact space is Eberlein compact}.
\newblock {\em Studia Mathematica}, 2(98):157--174, 1991.

\bibitem{zbMATH03465962}
Haskell~P. {Rosenthal}.
\newblock {A characterization of Banach spaces containing $l^1$.}
\newblock {\em {Proc. Natl. Acad. Sci. USA}}, 71:2411--2413, 1974.

\bibitem{zbMATH00515838}
C.~{Stegall}.
\newblock {Spaces of Lipschitz functions on Banach spaces.}
\newblock {\em {Functional analysis (Essen, 1991), Lecture Notes in Pure and
  Appl. Math.}}, 150:265--278, 1994.

\bibitem{Talagrand1979}
Michel Talagrand.
\newblock {Espaces de Banach Faiblement $\kappa$-Analytiques}.
\newblock {\em Annals of Mathematics}, 110(3):pp. 407--438, 1979.

\end{thebibliography}
\bibliographystyle{plain}

\end{document}